\documentclass[11pt,twoside,a4paper]{article}

\usepackage{amsmath,amsfonts,euscript,amssymb,amsthm,a4,times}

\def\real{\mathbb R}
\def\N{\mathbb N}
\def\Rn{{\mathbb R}^{n}}
\def\R{{\mathbb R}}
\def\RN{{\mathbb R}^{N}}
\def\M{{\mathbb R}^{N\times n}}

\def\proofof#1{\begin{proof}[Proof of #1]}

\def\dx{{\mathrm d}x}

\def\loc{\rm loc}

\def\dist{\mbox{dist }}
\def\eps{\varepsilon}
\def\teps{\tilde{\varepsilon}}
\def\tu{\tilde{u}_{\varepsilon}}
\def\ue{u_{\varepsilon}}
\def\wei{\langle Du_{\varepsilon} \rangle}

\catcode`@=11
\newbox\tr@tto
\setbox\tr@tto=\hbox{{\count0=0\dimen0=-,9pt\dimen1=1,1pt\loop\ifnum\count0<11 \advance \count0 by1 \vrule width.51pt height\dimen1
     depth\dimen0\kern-0.17pt\advance\dimen0 by-0.05pt\advance\dimen1
     by0.1pt\repeat \loop\ifnum\count0<21\advance \count0 by1 \vrule
     width.6pt height\dimen1 depth\dimen0\kern-0.2pt \advance\dimen0
     by-0.1pt\advance\dimen1 by 0.05pt\repeat}}
\def\medint{\displaystyle\copy\tr@tto\kern-10.4pt\int}
\catcode`@=12

\newtheorem{theorem}{Theorem}
\newtheorem{proposition}[theorem]{Proposition}
\newtheorem{lemma}[theorem]{Lemma}

\newtheorem{remark}[theorem]{Remark}
\newtheorem{definition}{Definition}

\newcommand{\BB}{\mathrm{B}}
\newcommand{\LL}{\mathrm{L}}
\newcommand{\WW}{\mathrm{W}}

\newcommand{\CC}{\mathrm{C}}

\newcommand{\dd}{{\rm d}}

\numberwithin{equation}{section}

\title{Regularity of minimizers of autonomous convex variational integrals}
\author{Menita Carozza, Jan Kristensen and Antonia Passarelli di Napoli\\
Accepted for publication in Ann.~Scuola Norm.~Sup.~Pisa Cl.~Sci.\\
DOI Number: 10.2422/2036-2145.201208$\mbox{}_{-}$005}

\date{September 2012}

\begin{document}

\maketitle

\begin{abstract}
We establish local higher integrability and differentiability results for minimizers of
variational integrals
$$
\mathfrak{F}( v, \Omega )= \int_{\Omega} \! F(Dv(x)) \, \dx
$$
over $\WW^{1,p}$--Sobolev mappings $v\colon \Omega \subset \Rn \to \RN$ satisfying a
Dirichlet boundary condition. The integrands $F$ are assumed to be autonomous, convex and of
$(p,q)$ growth, but are otherwise not subjected to any further structure conditions, and
we consider exponents in the range $1<p \leq q <p^{\ast}$, where $p^\ast$ denotes the Sobolev
conjugate exponent of $p$.
\bigskip

\bigskip


\end{abstract}

\maketitle

\section{Introduction  and Statement of Results }

We prove local higher integrability and differentiability results for minimizers of the
basic autonomous and convex variational integrals
\begin{equation}\label{defint}
\mathfrak{F}( v, O )= \int_{O} \! F(Dv(x)) \, \dx
\end{equation}
with integrands $F$ satisfying $(p,q)$ growth conditions, defined for mappings
$v\colon \Omega \to \RN$ of Sobolev class $\WW^{1,p}$ and open subsets $O$
of a fixed bounded and open subset $\Omega$ of $\Rn$. Our results mainly concern the
multi--dimensional vectorial case $n$, $N \geq 2$, but some aspects are also new 
in the multi--dimensional scalar case, $n \geq 2$ and $N=1$. The one--dimensional case, $n=1$
and $N \geq 1$, is special and stronger results apply there.

In order to state and describe the results we consider an integrand $F\colon \M \to \real$ satisfying 
the following growth and convexity hypotheses:
$$
F(\xi)\le L( |\xi |^{q}+1)  \eqno{\rm (H1)}
$$
$$
\xi \mapsto F(\xi ) - \ell \Bigl( \mu^{2}+| \xi |^{2}\Bigr)^{\frac{p}{2}} \quad \mbox{ is convex } \eqno{\rm (H2)}
$$
for all $\xi \in \M$. Here $L$, $\ell >0$ and $\mu \geq 0$ are arbitrary but fixed constants, whereas the 
exponents $q \ge p > 1$ will be subjected to various constraints.

The hypothesis (H2) is a uniform strong $p$--convexity condition for the
integrand $F$, and is similar to the condition considered in \cite{ff}. In fact, when $F$ is $\CC^2$ then
(H2) is equivalent to the following standard strong $p$--ellipticity condition
$$
F^{\prime \prime}(\xi)[\eta,\eta ]\geq c \Bigl( \mu^{2}+| \xi |^{2} \Bigr)^{\frac{p-2}{2}}|\eta |^{2} \eqno{\rm (H2')}
$$
for all $\xi$, $\eta \in \M$, where $c$ is a positive constant of form $c=c(p)\ell$. While when $F$ is $\CC^1$ the
hypothesis (H2) is equivalent to the following standard strong $p$--monotonicity condition
$$
\bigl\langle F^{\prime}(\xi )-F^{\prime}(\eta ), \xi - \eta \bigr\rangle \geq
c \bigl( \mu^{2}+|\xi |^{2}+|\eta |^{2} \bigr)^{\frac{p-2}{2}}| \xi-\eta |^{2} \eqno{\rm (H2'')}
$$
for all $\xi$, $\eta \in \M$, where again $c=\tilde{c}(p)\ell$.
It is easy to see that hypothesis (H2) in particular implies that the integrand $F$
is bounded from below, and we shall therefore often implicitly assume that $F$ is nonnegative once (H2) holds.

It is well--known that for convex integrands, the growth condition (H1) implies a Lipschitz
condition, which for $\CC^1$ integrands can be stated as
\begin{equation}\label{(H4)}
|F^{\prime}(\xi)|\le c(|\xi |^{q-1}+1 )
\end{equation}
for all $\xi \in \M$, where we can use $c=2^{q}L$.

We will be concerned with two closely related classes of $F$--minimizers of the variational
integral (\ref{defint}), under the assumptions (H1), (H2), which are defined as follows.

\begin{definition}

\noindent
(i) A mapping $u \in \WW^{1,p}(\Omega ,\RN )$ is an $F$--minimizer if $F(Du) \in \LL^{1}(\Omega )$
and
$$
\int_{\Omega} \! F(Du) \leq \int_{\Omega} \! F(Dv)
$$
for all $v \in \WW^{1,p}_{u}(\Omega , \RN )$.

\noindent
(ii) A mapping $u\in \WW^{1,1}_{\rm loc}(\Omega, \RN )$ is a local $F$--minimizer if
$F(Du) \in \LL^{1}_{\rm loc}(\Omega )$ and
$$
\int_{O} \! F(Du) \leq \int_{O} \! F(Dv)
$$
for any $O \Subset \Omega$ and any $v\in \WW_{u}^{1,p}(O, \RN )$.
\end{definition}
\noindent
Here we use the notation $O \Subset \Omega$ as a short--hand for: $O$ is an open set whose closure, $\overline{O}$,
is compact and contained in $\Omega$. Furthermore for an open subset $O \subset \Omega$ (no regularity of $\partial O$ implied) 
and $u \in \WW^{1,p}(O, \RN )$ we denote by $\WW^{1,p}_{u}(O,\RN )$ the Dirichlet class of $\WW^{1,p}$ Sobolev maps $v$ such that the difference
$v-u \in \WW^{1,p}_{0}(O,\RN )$, where the latter is defined to be the closure of the space of smooth and compactly supported
test maps, $\CC^{\infty}_{c}(O, \RN )$, in $\WW^{1,p}(O, \RN )$. We refer to the monograph \cite{zi} for background theory
on the relevant function spaces.

\noindent
We emphasize that in the definitions of $F$--minimality it is crucial for regularity theory that we include the integrability
requirements on $F(Du)$.

The assumptions (H1), (H2) clearly entail a $(p,q)$ growth condition:
there exists a constant $c=c(L/\ell ,p,q,\mu )>0$ such that
\begin{equation}\label{p,q}
\frac{1}{c}|\xi |^{p} - c \leq F(\xi ) \leq c(|\xi |^{q} + 1)
\end{equation}
for all $\xi \in \M$.

A systematic study of the regularity of minimizers of
such functionals was initiated with the celebrated papers by Marcellini (see in particular \cite{m2,m3,m4}). From the beginning it 
has been clear that no regularity can be expected if the coercitivity and growth exponents, denoted $p$ and $q$,
respectively, are too far apart (see \cite{m1,g2,hong} and also \cite{elm2,fmm}).
On the other hand,  many regularity results are available if the ratio $q/p$ is bounded above by a
suitable constant that in general depends on the dimension $n$, and converges to $1$ when $n$ tends to infinity
(incl.~\cite{af1,b,CKPdN,elm0,elm1,elm2,mp,ps}).

In the present context of $p$--convex integrands of $q$--growth the higher differentiability of minimizers is obtained by a 
variant of the difference--quotient method in connection with some sort of regularization procedure (see in particular 
\cite{elm0,elm1,elm2} and \cite{CKPdN}). In particular we emphasize that all such results rely heavily on the strong convexity 
hypothesis imposed on the integrand, and that this allows for a treatment based on (a suitable version of) the Euler--Lagrange system. 
However it should be noted that for a direct derivation of the Euler--Lagrange system we already have to require that the exponents $p$ and $q$
are sufficiently close.  Indeed, it is
well--known that for a convex $\CC^1$ integrand $F$ satisfying (\ref{p,q}) with exponents $q \leq p+1$
the $F$--minimality of a $\WW^{1,p}$--map $u$ implies that it is an $F$--extremal too: the field
$F^{\prime}(Du)$ is locally integrable (in fact, locally $p/(q-1)$--integrable by (\ref{(H4)})) and is row--wise solenoidal
in the sense that
$$
\int_{\Omega} \! \langle F^{\prime}(Du) , D\varphi \rangle = 0
$$
for all $\varphi \in \CC^{\infty}_{c}(\Omega , \RN )$. Using a regularization procedure and convex duality theory we
shall  establish much stronger results in Proposition \ref{dual} below.


The main results of this paper are the local higher integrability result stated in Theorem \ref{th1} and the local higher
differentiability result stated in Theorem \ref{th2}. Both results concern minimizers of the integral functional \eqref{defint} under the
assumptions (H1) and (H2). We obtain local higher integrability for minimizers when the exponents $p$ and $q$ satisfy the condition $1<p \leq q <p^{\ast}$, 
where $p^{\ast}$ denotes the Sobolev conjugate exponent of $p$. The interpretation of this is 
\begin{equation}\label{interpret}
\begin{split}
1<p \leq q < \frac{np}{n-p} \mbox{ when } p<n\\
1<p \leq q < \infty \mbox{ when } p \geq n .
\end{split}
\end{equation}
We emphasize that as opposed to many of the above mentioned results, we do not require any additional 
structure assumption on the integrand. More precisely we have the following:

\begin{theorem}\label{th1}
Let $F \colon \M \to \R$ be $\CC^1$ and satisfy (H1), (H2) with $\mu =0$. Assume
\begin{equation}\label{qcond}
1<p \leq q < p^{\ast},
\end{equation}
where $p^{\ast}$ denotes the Sobolev conjugate exponent of $p$ (with the interpretation of (\ref{interpret})).
For $g \in \WW^{1,q}(\Omega ,\RN )$ let $u \in \WW^{1,p}_{g}(\Omega , \RN )$ be the
unique $F$--minimizer. Then $u \in \WW^{1,q}_{\rm loc}(\Omega , \RN )$ when $q<\frac{np}{n-1}$,
and $u \in \WW^{1,r}_{\rm loc}(\Omega , \RN )$ for all $r<\bar{p}$, where
\begin{equation}\label{myst}
\bar{p} := \frac{np}{n-\frac{p}{p-1}\Bigl( 1-n(\frac{1}{p}-\frac{1}{q})\Bigr)} \quad \mbox{ when }
\quad q \geq \frac{np}{n-1}.
\end{equation}
\end{theorem}
The proof of Theorem \ref{th1} is based on the difference--quotient method but, in contrast to
the above mentioned papers, under our assumptions on the exponents $p$ and $q$, we can not use
directly that an $F$--minimizer is a solution to the corresponding Euler--Lagrange system. Instead we 
approximate the integrand $F$ by suitably regular integrands in order to facilitate a systematic use of 
the dual problems in the sense of Convex Analysis. Namely we approximate $F$ by strictly convex and 
uniformly elliptic integrands $F_k$, satisfying standard $p$--growth conditions, whose minimizers $u_k$ strongly
converge to the minimizer $u$ in $\WW^{1,p}$. To every such minimizer $u_k$ we can associate, essentially according 
to the standard duality theory for convex problems \cite{et}, a row--wise solenoidal matrix field denoted by $\sigma_k$. 
For the pair $(Du_k, \sigma_k)$ we shall establish suitable estimates, that are preserved in passing to the limit.
Such estimates will provide conditions in order for the Euler--Lagrange system to hold for an $F$--minimizer,
as well as a first regularity result (see Proposition \ref{dual} in Section 3).
While the dual problems have been used previously in regularity theory, see for instance \cite{cmp}, \cite{cp},
it seems that the observations used here have so far not been employed in the anisotropic
growth context. We refer the reader to Remark \ref{remarque} for a discussion of the somewhat mysterious exponent $\bar{p}$
that appears in Theorem \ref{th1}.

\noindent
As a consequence of the higher integrability result of Theorem \ref{th1}, we are to able to
establish the following

\begin{theorem}\label{th2}
Assume $F \colon \M \to \R$ is $\CC^1$ and satisfies (H1), (H2) for some exponents
$1<p \leq q <\infty$. Let $u \in \WW^{1,p}_{\mathrm{loc}}(\Omega , \RN )$ be a local $F$--minimizer. Setting
\begin{equation*}
V( Du )  := \Bigl( \mu^{2}+| Du |^{2} \Bigr)^{\frac{p-2}{4}}Du ,
\end{equation*}
we have that $$V(Du) \in \WW^{1,2}_{\rm loc}(\Omega , \M ),$$
provided $q < \tfrac{np}{n-1}$.
\end{theorem}


\noindent
We remark that under the assumptions (H1), (H2) on the integrand $F$ the $\WW^{1,2}_{\rm loc}$ higher differentiability of
$V(Du)$ has previously been established only when the exponents $p$, $q$ satisfy the stronger conditions $1< p \leq q < \tfrac{n+1}{n}p$,
see \cite{elm1}, and also \cite{ff,elm0,elm2,b,CKPdN}. Our improved bound $q < \tfrac{np}{n-1}$ is
in a certain sense more natural as it also appears in connection with optimal trace and embedding results
for Sobolev--type spaces. Namely when $B \subset \Rn$ is an open ball (or any smooth bounded domain), then any 
$h \in \WW^{1,p}(\partial B)$ can be extended to $H \in \WW^{1,q}(B)$ precisely 
for $q \leq \tfrac{np}{n-1}$. See for instance \cite{tar} for such trace and embedding theorems.

\noindent
Finally, we remark that our results can be generalized to minimizers of the general autonomous
convex variational integral
$$
\int_{\Omega} \! F(v,Dv) \, \dx .
$$
The precise statements and proofs are left to the interested reader, and we only note that
it is essential that the integrand $F=F(y,z)$ be jointly convex for the validity of the
results. Similar remarks, together with precise statements and sketches of proofs, were given
in \cite{CKPdN}.

\section{Preliminaries}\label{sectpre}

In this paper we follow the usual convention and denote by $c$ a
general constant that may vary on different occasions, even within the
same line of estimates.
Relevant dependencies on parameters and special constants will be suitably emphasized using
parentheses or subscripts. All the norms we use
on $\Rn$, $\RN$ and $\M$ will be the standard euclidean ones and denoted
by $| \cdot |$ in all cases. In particular, for matrices $\xi$, $\eta \in \M$ we write
$\langle \xi , \eta \rangle := \mbox{trace}(\xi^{T}\eta )$ for the usual inner product
of $\xi$ and $\eta$, and $| \xi | := \langle \xi , \xi \rangle^{\frac{1}{2}}$ for the
corresponding euclidean norm.
When $a \in \RN$ and $b \in \Rn$ we write $a \otimes b \in \M$ for the tensor product defined
as the matrix that has the element $a_{r}b_{s}$ in its r-th row and s-th column. Observe that
$(a \otimes b)x = (b \cdot x)a$ for $x \in \Rn$, and $| a \otimes b | = |a| |b|$.

When $F \colon \M \to \R$ is sufficiently differentiable, we write
$$
F^{\prime}(\xi )[\eta ] := \frac{\rm d}{{\rm d}t}\Big|_{t=0} F(\xi +t\eta )
\quad \mbox{ and } \quad
F^{\prime \prime}(\xi )[\eta ,\eta ] := \frac{\rm d^2}{{\rm d}t^{2}}\Big|_{t=0} F(\xi +t\eta )
$$
for $\xi$, $\eta \in \M$. Hereby we think of $F^{\prime}(\xi )$ both as an $N \times n$ matrix
and as the corresponding linear form on $\M$, though $|F^{\prime}(\xi )|$ will always
denote the euclidean norm of the matrix $F^{\prime}(\xi )$. The second derivative,
$F^{\prime \prime}(\xi )$, is a real bilinear form on $\M$. We express growth conditions for
the second derivative of the integrand in terms of the operator norm on bilinear forms:
$$
\| F^{\prime \prime}(\xi ) \| := \sup_{| \eta | \leq 1, |\zeta | \leq 1}
F^{\prime \prime}(\xi )[\eta , \zeta ] .
$$
It is convenient, and by now common, to express the convexity and growth conditions for the
integrands in terms of two auxiliary functions defined for all $\xi \in \M$ as
\begin{equation}\label{aux1}
\langle \xi \rangle = \langle \xi \rangle_{\mu} := \Bigl( \mu^{2}+| \xi |^{2} \Bigr)^{\frac{1}{2}}
\end{equation}
and
\begin{equation}\label{aux2}
V( \xi ) = V_{p,\mu}(\xi ) := \Bigl( \mu^{2}+| \xi |^{2} \Bigr)^{\frac{p-2}{4}}\xi ,
\end{equation}
where $\mu \geq 0$ and $p \geq 1$ are parameters.
For the auxiliary function $V_{p,\mu}$, we record the following
estimate (see the proof of \cite[Lemma 8.3]{gi}):

\begin{lemma} \label{V}
Let $1<p<\infty$ and $0\leq \mu \leq 1$. There exists a constant $c=c(n,N,p)>0$
such that
$$
c^{-1}\Bigl( \mu^{2}+| \xi |^2+| \eta |^2 \Bigr)^{\frac{p-2}{2}}\leq
\frac{|V_{p,\mu}(\xi )-V_{p,\mu}(\eta )|^2}{|\xi -\eta |^2} \leq
c\Bigl( \mu^{2}+|\xi |^2+|\eta |^2 \Bigr)^{\frac{p-2}{2}}
$$
for any $\xi$, $\eta \in \M$.
\end{lemma}

\noindent
We shall use a class of fractional Sobolev spaces that can be defined in terms of Nikolskii
conditions. For a vector valued function
$w \colon A \to \R^k$, a natural number
$1 \leq s \leq n$ and a real number $h \in \R$, we define
the finite difference operator
$$
\Delta_{s,h}w(x) :=  w(x+he_{s})-w(x),
$$
where $\{e_{1}, \, \dots \, , \, e_{n}\}$ denotes the canonical basis  of $\Rn$.
Note that hereby $\Delta_{s,h}w(x)$ is well--defined whenever $x, x +he_{s} \in A$.

\begin{definition}
Let $A \subset \Rn$ be an open set, $k \in \N$, $\alpha \in (0,1)$ and $q \in [1,\infty )$.
For a mapping $w \in \LL^{q}_{\rm loc}(A,\R^k )$ we say that $w$ is locally in
$\BB^{\alpha ,q}_{\infty}$ on $A$ provided for each ball $B \Subset A$ there exist
$d \in (0,\mathrm{dist} (B,\partial A))$, $M > 0$ such that
$$
\int_{B} \! |\Delta_{s,h}w(x)|^q \, \dd x  \leq M |h|^{q\alpha}
$$
for every $s \in \{1,\ldots,n\}$ and $h\in \R$ satisfying
$|h|\leq d$.
\end{definition}

\begin{theorem}\label{nikol}
On any domain $\Omega \subset \Rn$ we have the continuous embeddings:
\begin{itemize}
\item[(i)] $\BB^{\alpha ,q}_{\infty} \hookrightarrow \LL^{r}_{\rm loc}$ \, for all $r< \frac{nq}{n-\alpha q}$ 
provided $\alpha \in (0,1)$, $q>1$ and $\alpha q < n$;
\item[(ii)] $\WW^{1,p}_{\rm loc} \hookrightarrow \BB^{\alpha , q}_{\infty}$ provided $\alpha = 1-n(\tfrac{1}{p}-\tfrac{1}{q})$, where
$1<p \leq q < \infty$.
\end{itemize}
\end{theorem}
\noindent 
We refer to sections 30--32 in \cite{tar} for a proof of this theorem. In fact, the above statements follow by localizing the
corresponding results proved for functions defined on $\Rn$ in \cite{tar} by simply using a smooth cut--off function.

We shall require some further elementary notions from convex analysis, all of which are
discussed in the scalar case $N=1$ in \cite{et}. However, as we shall briefly demonstrate below,
the relevant parts easily extend to the vectorial case $N>1$ too.
Let $F\colon \M \to \R$ satisfy the $(p,q)$ growth condition:
$$
c_{1}|\xi |^{p}-c_{2} \leq F(\xi ) \leq c_{2}(|\xi |^{q}+1),
$$
where $0<c_{1} \leq c_{2}$ and $1<p\leq q < \infty$. Its polar (or Fenchel conjugate)
integrand is defined by
\begin{equation}\label{polardef}
F^{\ast}(\zeta ) := \sup_{\xi \in \M} \Bigl( \langle \zeta , \xi \rangle - F(\xi )
\Bigr) , \quad \zeta \in \M .
\end{equation}
Hereby $F^{\ast}\colon \M \to \R$ is convex and has $(q^{\prime},p^{\prime})$
growth, where $p^{\prime}$, $q^{\prime}$ are the H\"{o}lder conjugate exponents of
$p$, $q$, respectively. More precisely, as can readily be checked, we have
\begin{equation}\label{polarpq}
c_{3}|\zeta |^{q^{\prime}}-c_{2} \leq F^{\ast}(\zeta ) \leq c_{4}|\zeta |^{p^{\prime}}+c_{2}
\end{equation}
for all $\zeta \in \M$, where $c_{3} = c_{2}^{-\frac{1}{q-1}}(1-\frac{1}{q})q^{-\frac{1}{q-1}}$ and
$c_{4}=c_{1}^{-\frac{1}{p-1}}(1-\frac{1}{p})p^{-\frac{1}{p-1}}$. One can check that the
bipolar integrand $F^{\ast \ast} := (F^{\ast})^{\ast}$ equals $F$ at $\xi$ if and only if $F$
is lower semicontinuous and convex at $\xi$, and more generally, that it is the convex
envelope of $F$. In particular, $F^{\ast \ast} = F$ precisely when $F$ is convex and lower
semicontinuous (the latter being a consequence of the former when, as here, $F$ is
real--valued).

The definition of polar integrand means that we have the Young--type inequality
\begin{equation}\label{young}
\langle \zeta , \xi \rangle \leq F^{\ast}(\zeta ) + F^{\ast \ast}(\xi )
\end{equation}
for all $\zeta$, $\xi \in \M$. Notice that for a given $\xi$ we have equality
in (\ref{young}) precisely for $\zeta \in \partial F^{\ast \ast}(\xi )$, the subgradient
for $F^{\ast \ast}$ at $\xi$. Furthermore, we record that $F$ is strictly convex precisely
when $F^{\ast}$ is $\CC^1$, and that in this case we also have
\begin{equation}\label{c1}
(F^{\ast})^{\prime}(F^{\prime}(\xi )) = \xi
\end{equation}
for all $\xi \in \M$.

We now specialize to integrands satisfying standard $p$--growth and convexity conditions,
and so assume that $F\colon \M \to \R$ is a $\CC^1$ function satisfying
\begin{equation}\label{pgr1}
|F(\xi )| \leq L(|\xi |^p +1)
\end{equation}
and
\begin{equation}\label{pconvex1}
\xi \mapsto F(\xi )-\ell |\xi |^{p} \quad \mbox{ is convex,}
\end{equation}
where $0< \ell \leq L < \infty$ and $1<p<\infty$. Note that (\ref{pgr1}) is (H1)
with $\mu = 1$ ($q=p$ and slightly larger $L$) and (\ref{pconvex1}) is (H2) with $\mu =0$.
The polar integrand $F^{\ast}\colon \M \to \R$
is then strictly convex and $\CC^1$, and it satisfies the $p^{\prime}$--growth condition:
$$
c_{1}|\zeta |^{p^{\prime}} - c_{2} \leq F^{\ast}(\zeta ) \leq c_{2}(|\zeta |^{p^{\prime}}+1)
$$
for all $\zeta \in \M$, where $c_{1}=c_{1}(L,p)> 0$, $c_{2}=c_{2}(\ell ,p) > 0$ and
$p^{\prime}=p/(p-1)$. By standard arguments, for a given $g \in \WW^{1,p}(\Omega , \RN )$, the problem
of minimizing $\int_{\Omega} \! F(Dv)$ over $v \in \WW^{1,p}_{g}(\Omega , \RN )$ admits a
unique solution $u$. This minimizer is also the unique weak solution
$u \in \WW^{1,p}_{g}(\Omega , \RN )$ to the Euler--Lagrange equation:
$$
\int_{\Omega} \! \langle F^{\prime}(Du),D\varphi \rangle = 0
$$
for all $\varphi \in \WW^{1,p}_{0}(\Omega , \RN )$. In view of the Young--type inequality
(\ref{young}), and the subsequent remark, we have for the minimizer $u$
the extremality relation:
\begin{equation}\label{extremality}
\langle F^{\prime}(Du),Du \rangle = F^{\ast}(F^{\prime}(Du))+F(Du)
\end{equation}
valid pointwise almost everywhere on $\Omega$. Hence for any row--wise solenoidal field
$\sigma \in \LL^{p^{\prime}}(\Omega , \M )$ it follows that
$$
\int_{\Omega} \Bigl( \langle F^{\prime}(Du),Du \rangle -F^{\ast}(F^{\prime}(Du)) \Bigr)
\geq \int_{\Omega} \Bigl( \langle \sigma ,Du \rangle -F^{\ast}(\sigma ) \Bigr) .
$$
Because $u-g \in \WW^{1,p}_{0}(\Omega , \RN )$ and $\sigma$ is row--wise solenoidal
and $p^{\prime}$--integrable, we have $\int_{\Omega} \! \langle \sigma ,Du \rangle =
\int_{\Omega} \! \langle \sigma ,Dg \rangle$. Consequently, $F^{\prime}(Du)$ is the unique
maximizer of the functional
\begin{equation}\label{dualfunct}
\sigma \mapsto \int_{\Omega} \! \Bigl( \langle \sigma , Dg \rangle - F^{\ast}(\sigma ) \Bigr)
\end{equation}
over all row--wise solenoidal fields $\sigma \in \LL^{p^{\prime}}(\Omega , \M )$. The extremality
relation (\ref{extremality}) can also be expressed in terms of $\sigma^{\ast} := F^{\prime}(Du)$
and then reads as
\begin{equation}\label{extra}
\langle \sigma^{\ast},(F^{\ast})^{\prime}(\sigma^{\ast}) \rangle = F^{\ast}(\sigma^{\ast})
+ F((F^{\ast})^{\prime}(\sigma^{\ast})),
\end{equation}
where $(F^{\ast})^{\prime}(\sigma^{\ast})=Du$ being row--wise curl--free is merely
a restatement of the Euler--Lagrange equation for the maximization problem of the
functional (\ref{dualfunct}) over solenoidal fields.

\section{Proof of Theorem \ref{th1}}

Throughout this section we let $F \colon \M \to \R$ be a $\CC^1$ integrand
satisfying
\begin{equation}\label{pconv}
\xi \mapsto F(\xi )-\ell |\xi |^p \quad \mbox{ is convex }
\end{equation}
and
\begin{equation}\label{qgrow}
0 \leq F(\xi ) \leq c(|\xi |^{q}+1)
\end{equation}
where $1<p \leq q < \infty$ and $\ell$, $c > 0$ are constants.
We shall impose additional conditions on the exponents $p$ and $q$ as we go along.
It is a routine matter to check that $\mathfrak{F}( v, \Omega ) =\int_{\Omega} \! F(Dv)$ under the
assumptions (\ref{pconv}), (\ref{qgrow}) is a strictly convex, lower semicontinuous and proper
functional on $\WW^{1,p}(\Omega , \RN )$. Hence for a given $g \in \WW^{1,p}(\Omega , \RN )$ with
$\mathfrak{F}(g,\Omega ) < \infty$ the existence and uniqueness of a minimizer $u$ in the Dirichlet
class $\WW^{1,p}_{g}(\Omega , \RN )$ is then evident.

We split the proof of Theorem \ref{th1} in two parts and start with the following preliminary result. 
We state it as a separate result because we believe it could have independent interest.
It gives conditions for the Euler--Lagrange equation to hold for the minimizer $u$, and also contains a first
higher integrability result. 

\begin{proposition}\label{dual}
Assume $F \colon \M \to \R$ is $\CC^1$ and satisfies (\ref{pconv}), (\ref{qgrow}).
For $g \in \WW^{1,p}(\Omega ,\RN )$ with $F(Dg) \in \LL^{1}(\Omega )$, let
$u \in \WW^{1,p}_{g}(\Omega , \RN )$ denote the unique $F$--minimizer. We then
have the following two statements (where $F^\ast$ denotes the polar of $F$):
\begin{itemize}
\item[(i)] If $g \in \WW^{1,q}(\Omega , \RN )$, then $F^{\ast}(F^{\prime}(Du)) \in \LL^{1}(\Omega )$
and $F^{\prime}(Du)$ is row--wise solenoidal.
\item[(ii)] If $q \leq \frac{np}{n-1}$, then $F^{\ast}(F^{\prime}(Du)) \in \LL^{1}_{\rm loc}(\Omega )$
and $F^{\prime}(Du)$ is row--wise solenoidal.
\end{itemize}
Hence in both cases (i)--(ii), $u$ is in particular an $F$--extremal
and $F^{\prime}(Du) \in \LL^{q^{\prime}}(\Omega , \M )$, where $q^{\prime}=q/(q-1)$.
\end{proposition}

\begin{remark}\label{dualrem}
Note that in (i) above, apart from $1 < p \leq q < \infty$, we do not impose any conditions on
the exponents $p$ and $q$.
Furthermore, as the reader can observe from the proof below, we in fact establish that
$\sigma^{\ast} := F^{\prime}(Du)$ is the unique maximizer of the dual problem that consists in
maximizing the functional
$$
\sigma \mapsto \int_{\Omega} \! \Bigl( \langle Dg , \sigma \rangle - F^{\ast}(\sigma ) \Bigr)
$$
over row--wise solenoidal fields $\sigma \in \LL^{q^{\prime}}(\Omega , \M )$, where $F^{\ast}$
denotes the polar integrand of $F$ and $q^{\prime}=q/(q-1)$.
\end{remark}

\begin{proof}
We start by constructing a class of auxiliary problems, whose solutions on the one
hand approximate the minimizer $u$, and on the other can be dealt with by standard means.
The construction might seem a bit elaborate at first, but all properties that we establish
play a role in the proof, though some only at a later stage of the proof of the main result.
Put $$G(\xi ) := F(\xi ) -\frac{\ell}{2}|\xi |^{p},$$ $\xi \in \M.$ Then
$G-\frac{\ell}{2}|\cdot |^p$ is convex by (\ref{pconv}) so that in particular
$G(\xi ) -\frac{\ell}{2}|\xi |^{p} \geq G(0)+\langle G^{\prime}(0),\xi \rangle$ for
all $\xi$, and hence invoking also (\ref{qgrow}) we find a (new) constant $c>0$ such
that
\begin{equation}\label{ggrow}
c(|\xi |^{q}+1) \geq G(\xi ) \geq \frac{1}{c}|\xi |^{p}-c
\end{equation}
for all $\xi \in \M$. Consequently, the polar of $G$,
$$
G^{\ast}(z) := \sup_{\xi \in \M} \Bigl( \langle \xi ,z \rangle - G(\xi ) \Bigr)
$$
is a real--valued convex function satisfying a $(q^{\prime},p^{\prime})$ growth
condition, where $q^{\prime}$, $p^{\prime}$ denote the H\"{o}lder conjugate exponents of
$q$, $p$, respectively (compare with (\ref{polarpq})).

For each $k >0$ define
$$
G_{k}(\xi ) := \max_{|z| \leq k} \Bigl( \langle \xi ,z \rangle -G^{\ast}(z) \Bigr) .
$$
Then $G_{k}$ is a real--valued convex, globally $k$--Lipschitz function, and because $G$
is (lower semicontinuous and) convex we have that
$$
G_{k}(\xi ) \nearrow G^{\ast \ast}(\xi ) = G(\xi ) \quad \mbox{ as } \quad
k \nearrow \infty
$$
pointwise in $\xi \in \M$. Define
$$\tilde{G}_{k}(\xi ) := \max \{ G_{k}(\xi ),\frac{1}{c}|\xi |^{p}-c \}.$$ In view of
(\ref{ggrow}) we still have that $\tilde{G}_{k}(\xi ) \nearrow G(\xi )$ as
$k \nearrow \infty$. Since $p>1$ and $G_k$ is $k$--Lipschitz there exist numbers
$r_k > 0$ such that $r_{k} \nearrow \infty$ as $k \nearrow \infty$ and
$\tilde{G}_{k}(\xi ) = \frac{1}{c}|\xi |^{p}-c$ for $|\xi | \geq r_{k}-1$. Define
$$
H_{k}(\xi ) := \left\{
\begin{array}{ll}
\tilde{G}_{k}(\xi ) & \mbox{ when } |\xi | \leq r_{k}\\
\frac{p}{c}r_{k}^{p-1}|\xi | - \frac{p-1}{c}r_{k}^{p}-c & \mbox{ when } |\xi | > r_{k}.
\end{array}
\right.
$$
It is not hard to check that $H_k$ is convex and globally $m_k$--Lipschitz (we may
take any $m_{k} \geq \frac{p}{c}r_{k}^{p-1}$). Moreover,
\begin{equation}\label{fenapprox}
H_{k}(\xi ) \nearrow G^{\ast \ast}(\xi ) = G(\xi ) \quad \mbox{ as } \quad
k \nearrow \infty
\end{equation}
pointwise in $\xi \in \M$. Next we regularize $H_k$ by use of the
following standard radially symmetric and smooth convolution kernel
$$
\Phi (\xi ) := \left\{
\begin{array}{ll}
c \exp \left( \frac{1}{|\xi |^{2}-1} \right) & \mbox{ for } |\xi | < 1\\
0 & \mbox{ for } | \xi | \geq 1 ,
\end{array}
\right.
$$
where the constant $c=c(nN)$ is chosen such that $\int_{\M} \! \Phi = 1$, and
for each $\varepsilon > 0$ we put
$\Phi_{\varepsilon}(\xi ) := \varepsilon^{-nN}\Phi (\varepsilon^{-1}\xi )$.
It is routine to check that the mollified function $\Phi_{\varepsilon} \ast H_{k}$
(as usual defined by convolution) is convex and $\CC^{\infty}$, and since
$H_k$ is convex and $m_k$--Lipschitz that
\begin{equation}\label{klip}
H_{k}(\xi ) \leq (\Phi_{\varepsilon} \ast H_{k})(\xi ) \leq H_{k}(\xi ) +m_{k}\varepsilon
\end{equation}
holds for all $\xi \in \M$. For integers $k>1$ and sequences $(\delta_{k})$,
$(\mu_{k}) \subset (0,\infty )$ (specified at (\ref{delmy}) below), we define
\begin{equation}\label{fk}
F_{k}(\xi ) := (\Phi_{\delta_{k}} \ast H_{k})(\xi ) - \mu_{k}+
\frac{\ell}{2}|\xi |^{p} .
\end{equation}
Then we have for all $\xi \in \M$ and $k>1$:
\begin{eqnarray*}
F_{k}(\xi ) &\leq& H_{k}(\xi ) +m_{k}\delta_{k}-\mu_{k}+\frac{\ell}{2}|\xi |^{p}\\
           &\leq& H_{k+1}(\xi ) +m_{k}\delta_{k}-\mu_{k}+\frac{\ell}{2}|\xi |^{p}\\
&\leq& (\Phi_{\delta_{k+1}}\ast H_{k+1})(\xi ) +m_{k}\delta_{k}-\mu_{k}+ \frac{\ell}{2}|\xi |^{p}\\
&=& F_{k+1}(\xi )+\mu_{k+1}+m_{k}\delta_{k}-\mu_{k}.
\end{eqnarray*}
where we used \eqref{klip} and the monotonicity of the sequence $H_k$.
Hence taking
\begin{equation}\label{delmy}
\delta_{k} := \frac{1}{k^{2}m_{k}} \mbox{ and } \mu_{k}:= \frac{1}{k-1}
\end{equation}
 we have that $F_{k}(\xi ) \nearrow F(\xi )$ as $k \nearrow \infty$ pointwise in $\xi$. It follows
in particular from Dini's Lemma that the convergence is locally uniform in $\xi$.
We note that $F_k$ is $\CC^1$ on $\M$, $\CC^{\infty}$ on $\M \setminus \{ 0 \}$, and
that it is $\CC^2$ on $\M$ when $p \geq 2$.
Next we check that also $F_{k}^{\prime}(\xi ) \to F^{\prime}(\xi )$ locally uniformly in $\xi$
as $k \to \infty$. To that end assume that $\xi_k \to \xi$ and consider $(F_{k}^{\prime}(\xi_{k}))$.
Because difference--quotients of convex functions are increasing in the increment, we have for
all $\eta \in \M$ and $0< |t| \leq 1$:
\begin{eqnarray*}
\Bigl| \langle F_{k}^{\prime}(\xi_{k})-F^{\prime}(\xi),\eta \rangle \Bigr| &\leq& \left|
\frac{F_{k}(\xi_{k}+t\eta )-F_{k}(\xi_{k})-\langle F^{\prime}(\xi ),t\eta \rangle}{t}\right|\\
&\leq& \Bigl| F_{k}(\xi_{k}+\eta )-F_{k}(\xi_{k})-\langle F^{\prime}(\xi ),\eta \rangle \Bigr| .
\end{eqnarray*}
Consequently, we have for all $\eta \in \M$ that
$$
\limsup_{k \to \infty}\Bigl| \langle F_{k}^{\prime}(\xi_{k})-F^{\prime}(\xi),\eta \rangle \Bigr|
\leq \Bigl| F(\xi +\eta )-F(\xi )-\langle F^{\prime}(\xi ),\eta \rangle \Bigr| ,
$$
and since $F$ in particular is differentiable at $\xi$ we conclude that the left--hand side
must vanish. This proves the asserted local uniform convergence of derivatives.
Finally, we also record that $F_{k}-\frac{\ell}{2}|\cdot |^p$ is convex, that
\begin{equation}\label{d2fk}
\| F^{\prime \prime}_{k}(\xi ) \| \leq c_{k}(|\xi |^{p-2}+1)
\end{equation}
for all $\xi \in \M \setminus \{ 0 \}$, where $c_{k}$ are positive real constants (possibly
$c_{k} \nearrow \infty$ of course). It is also easy to see that $F_k$ satisfy a uniform $(p,q)$
growth condition.

Let $u_{k} \in \WW^{1,p}_{g}(\Omega ,\RN )$ denote the unique
$F_k$--minimizer, and recall from above that the row--wise solenoidal matrix field
$\sigma_{k} := F^{\prime}_{k}(Du_{k}) \in \LL^{p^{\prime}}(\Omega , \M )$ is a solution to
the dual problem that consists in maximizing the functional
$$
\int_{\Omega} \! \Bigl( \langle \sigma , Dg \rangle - F_{k}^{\ast}(\sigma ) \Bigr)
$$
over row--wise solenoidal matrix fields $\sigma \in \LL^{p}(\Omega , \M )$,
where $F_{k}^{\ast}$ denotes the polar of $F_k$.
As the $F_k$ satisfy a uniform
$(p,q)$ growth condition, the $F_{k}^{\ast}$ satisfy a uniform $(q^{\prime},p^{\prime})$
growth condition, and it is not difficult to check that
$F^{\ast}_{k}(\zeta ) \searrow F^{\ast}(\zeta )$ as $k \nearrow \infty$ pointwise in $\zeta$.
Furthermore, we record the extremality relation
\begin{equation}\label{extrerel}
\langle \sigma_{k},Du_{k} \rangle = F_{k}^{\ast}(\sigma_{k})+F_{k}(Du_{k})
\quad \mbox{ a.e.~on } \Omega
\end{equation}
that holds for all $k>1$, and that, since $\sigma_k \in \LL^{p^{\prime}}$ is row--wise
solenoidal and $u_{k}-g \in \WW^{1,p}_{0}(\Omega ,\RN )$,
\begin{equation}\label{extrerela}
\int_{\Omega} \! \langle \sigma_{k},Du_{k} \rangle =
\int_{\Omega} \! \langle \sigma_{k},Dg \rangle .
\end{equation}
Our next goal is to show that $u_{k} \to u$ strongly in $\WW^{1,p}$. To that end,
we start by observing that
$$
\int_{\Omega} \! \bigl( \frac{1}{c}|Du_{k}|^{p}-c \Bigr) \leq \int_{\Omega} \! F_{k}(Du_{k})
\leq \int_{\Omega} \! F(Dg) < \infty
$$
so $(u_{k})$ is bounded in $\WW^{1,p}(\Omega , \RN )$. Let $(u_{k'})$ be a subsequence.
Then by the reflexivity of $\WW^{1,p}$, it admits a further subsequence $(u_{k''})$ that
converges weakly to some $v$ in $\WW^{1,p}$. By Mazur's lemma, $\WW^{1,p}_{g}(\Omega , \RN )$
is also $\WW^{1,p}$--weakly closed, so $v \in \WW^{1,p}_{g}(\Omega , \RN )$, and
relabelling the subsequence we write simply $u_{k} \rightharpoonup v$. Now, by Mazur's Lemma,
we get for each $k>1$
$$
\liminf_{j \to \infty} \int_{\Omega} \! F_{k}(Du_{j}) \geq \int_{\Omega} \! F_{k}(Dv),
$$
and since $F_{k} \nearrow F$, we find by monotone convergence and minimality of $u$
$$
\liminf_{k \to \infty} \int_{\Omega} \! F_{k}(Du_{k}) \geq \int_{\Omega} \! F(Dv)
\geq \int_{\Omega} \! F(Du).
$$
Using first that $u_k$ is $F_k$--minimizing, and then monotone convergence, yield
$$
\limsup_{k \to \infty} \int_{\Omega} \! F_{k}(Du_{k}) \leq \limsup_{k \to \infty}
\int_{\Omega} \! F_{k}(Du) = \int_{\Omega} \! F(Du),
$$
and by comparing this with the foregoing inequalities we deduce that
\begin{equation}\label{strong}
\int_{\Omega} \! F_{k}(Du_{k}) \to \int_{\Omega} \! F(Du) = \int_{\Omega} \! F(Dv).
\end{equation}
By uniqueness of $F$--minimizers, $v=u$. To deduce that the convergence is actually
strong we use the uniform $p$--convexity of the $F_k$, we have that $F_{k}-\frac{\ell}{2}|\cdot |^p$
is convex for all $k>1$. So, as $F_k$ is $\CC^1$, we deduce from Lemma \ref{V} that
there exists a constant $c>0$ such that
$$
c|V(\xi ) - V(\eta ) |^{2} \leq F_{k}(\xi ) -F_{k}(\eta ) - \langle F_{k}^{\prime}(\eta ),
\xi -\eta \rangle
$$
holds for all $\xi$, $\eta \in \M$ and $k>1$. Here $V(\xi ) = V_{p,0}(\xi ) =
|\xi |^{\frac{p-2}{2}}\xi$. Consequently,
\begin{eqnarray*}
c\int_{\Omega} \! |V(Du)-V(Du_{k})|^{2} &\leq& \int_{\Omega} \! \Bigl( F_{k}(Du)-F_{k}(Du_{k})
-\langle F_{k}^{\prime}(Du_{k}),D(u-u_{k})\rangle \Bigr)\\
&=&  \int_{\Omega} \! \Bigl( F_{k}(Du)-F_{k}(Du_{k}) \Bigr) \to 0
\end{eqnarray*}
as $k \to \infty$. It follows that $Du_{k} \to Du$ in measure on $\Omega$ and that
$|V(Du_{k})|^{2}=|Du_{k}|^{p}$ is equi--integrable on $\Omega$, hence, by Vitali's
convergence theorem, that $Du_{k} \to Du$ strongly in $\LL^{p}$.
Since $u_{k}-u \in \WW^{1,p}_{0}(\Omega , \RN )$
we have shown that the (relabelled) subsequence $(u_{k})$ converges strongly to $u$ in
$\WW^{1,p}$.
By the uniqueness of limit we conclude by a standard argument that the full sequence $(u_k )$
converges strongly in $\WW^{1,p}$ to $u$.
It follows in particular that $\sigma_{k}=F_{k}^{\prime}(Du_{k}) \to F^{\prime}(Du)$
in measure on $\Omega$, and so passing to the limit in (\ref{extrerel}) we recover, with
$\sigma^{\ast} := F^{\prime}(Du)$, the pointwise extremality relation
\begin{equation}\label{pextrerel}
\langle \sigma^{\ast} ,Du \rangle = F^{\ast}(\sigma^{\ast} )+F(Du )
\quad \mbox{ a.e.~on } \Omega .
\end{equation}
Up to this point we have not used any of the conditions on the boundary datum $g$ or on the exponents
$p$, $q$ listed in (i)--(ii).
We now assume that $g \in \WW^{1,q}(\Omega , \RN )$ corresponding to (i). Then in view
of (\ref{extrerel}), (\ref{extrerela}) and the uniform $(q^{\prime},p^{\prime})$ growth
of $F_{k}^{\ast}$ we deduce that $(\sigma_{k})$ is bounded in $\LL^{q^{\prime}}(\Omega , \M )$.
Namely,
\begin{eqnarray}
\int_\Omega \bigl( \frac{1}{c}|\sigma_k|^{q^{\prime}}-c \bigr) &\leq& \int_\Omega F^*_k (\sigma_k)=
\int_\Omega \langle \sigma_k,Du_k\rangle -F_k (Du_k)\cr\cr
&=&\int_\Omega \langle \sigma_k,Dg\rangle -F_k (Du_k)\cr\cr
&\leq&
\frac{c}{2}\int_\Omega |\sigma_k|^{q^{\prime}}+c\int_\Omega |Dg|^{q}+\int_\Omega F_k (Du_k)
\end{eqnarray}
and hence
\begin{equation}
\int_\Omega |\sigma_k|^{q^{\prime}}\le c\left(\int_\Omega |Dg|^{q}+\int_\Omega F_k (Du_k)\right)
\end{equation}
Therefore, by Fatou's lemma and by \eqref{strong}, we have that
$$
\| F^{\prime}(Du) \|_{\LL^{q^{\prime}}}^{q^{\prime}} \leq \liminf_{k \to \infty}
\| \sigma_{k} \|_{\LL^{q^{\prime}}}^{q^{\prime}} \leq  c\left(\int_\Omega |Dg|^{q}+\int_\Omega F (Du)\right).
$$
As it is then clear that $F^{\prime}(Du)$ is row--wise solenoidal this proves (i).
Finally, regarding Remark \ref{dualrem}, note that in view of (\ref{extrerel}),
(\ref{extrerela}) the field $\sigma^{\ast}$ is a maximizer of
$$
\sigma \mapsto \int_{\Omega} \! \Bigl( \langle Dg , \sigma \rangle - F^{\ast}(\sigma ) \Bigr)
$$
over row--wise solenoidal fields $\sigma \in \LL^{q^{\prime}}(\Omega , \M )$. By strict
convexity it is then the unique such maximizer.

We next turn to (ii), and assume that $q \leq np/(n-1)$. Since $u \in \WW^{1,p}(\Omega , \RN )$
we can for each $x_0 \in \Omega$ find a ball $B=B(x_{0},R) \subset \Omega$ such
that $u|_{\partial B} \in \WW^{1,p}(\partial B, \RN )$, see for instance \cite{zi}. If $h$ denotes the harmonic
extension of $u|_{\partial B}$ to $B$, then it is well--known that $h \in \WW^{1,\frac{np}{n-1}}(B, \RN )$.
We can now repeat the above argument for (i), where this time we define the auxiliary minimizers $u_k$
with $B$, $h$ substituted for $\Omega$, $g$, respectively.


\end{proof}

Now, we are going to prove Theorem \ref{th1} from the Introduction.
The key new point in the proof is that we estimate the field $\sigma^{\ast}=F^{\prime}(Du)$ using
Proposition \ref{dual} rather than merely by use of the (consequence of the) growth condition (\ref{(H4)}).
The outcome is a better higher integrability estimate for the minimizer. The remaining parts of the proof
are standard in the present context, and consist of a difference--quotient argument applied in the
setting of fractional Sobolev spaces (compare for instance \cite{elm2}) and the regularized problems
defined in the proof of Proposition \ref{dual}.

\begin{proof}[Proof of Theorem \ref{th1}: Conclusion]
Define the integrands $F_k$ and corresponding $F_k$--minimizers
$u_{k}$ of class $\WW^{1,p}_{g}(\Omega ,\RN )$ as in the proof of Proposition \ref{dual} (see
in particular (\ref{fk})). We have shown there that $u_{k} \to u$ strongly in $\WW^{1,p}$
and that $\sigma_{k} := F_{k}^{\prime}(Du_{k}) \to F^{\prime}(Du)$ weakly in $\LL^{q^{\prime}}$ and
in measure on $\Omega$. Furthermore, $F_{k}-\frac{\ell}{2}| \cdot |^{p}$ is convex
and
$$
\int_{\Omega} \! \langle F_{k}^{\prime}(Du_{k}),D\varphi \rangle = 0
$$
for all $\varphi \in \WW^{1,p}_{0}(\Omega , \RN )$. We shall establish uniform
integrability bounds on $(u_{k})$ to conclude the proof. Recall that the auxiliary $V$--function
for the degenerate case $\mu =0$ is defined as
$V(Du_{k})=V_{p,0}(Du_{k})=|Du_{k}|^{\frac{p-2}{2}}Du_{k}$.

Fix $B_{3R} = B(x_{0},3R) \subset \Omega$, an integer $1 \leq s \leq n$, and an increment
$0 \neq h \in (-R,R)$. It follows that
\begin{equation}\label{eldif}
\int_{B_{2R}} \! \langle \Delta_{s,h}F^{\prime}_{k}(Du_{k}),D\varphi \rangle = 0
\end{equation}
for all $\varphi \in \WW^{1,p}_{0}(B_{2R},\RN )$. In particular we may take
$\varphi = \theta \Delta_{s,h}u_{k}$ for $\theta \in \CC^{1}_{c}(B_{2R})$, whereby
$$
\int_{B_{2R}} \! \langle \Delta_{s,h}F^{\prime}_{k}(Du_{k}),\Delta_{s,h}Du_{k}
\rangle \theta = -\int_{B_{2R}} \! \langle \Delta_{s,h}F^{\prime}_{k}(Du_{k}),\Delta_{s,h}u_{k}
\otimes D\theta \rangle
$$
follows. Consequently, taking $\theta$ nonnegative and so $\theta = 1$ on $B_R$ we may use (H2'') (since $F_{k}$ is $\CC^{1}$),
Lemma 3 and H\"{o}lder's inequality to find a constant $c>0$, which in particular is independent
of $k$, such that
\begin{eqnarray}\label{step}
\int_{B_{R}} \! |\Delta_{s,h}V(Du_{k})|^{2} &\leq& c\left( \int_{B_{3R}} \!
|\sigma_{k}|^{q^{\prime}}\right)^{\frac{1}{q^{\prime}}}\left( \int_{B_{2R}}
\! |\Delta_{s,h}u_{k}|^{q} \right)^{\frac{1}{q}}\sup |D\theta | \nonumber\\
&\leq& \tilde{c}\left( \int_{B_{2R}} \! |\Delta_{s,h}u_{k}|^{q} \right)^{\frac{1}{q}},
\end{eqnarray}
where
$$
\tilde{c} := c \sup_{k}\left( \int_{B_{3R}} \!
|\sigma_{k}|^{q^{\prime}}\right)^{\frac{1}{q^{\prime}}}\sup |D\theta |
$$
is finite according to Proposition \ref{dual}. Now to extract information from this
estimate, we recall that $(u_{k})$ in particular is bounded in $\WW^{1,p}$, and that
by the version (ii) of the Sobolev Embedding stated in Theorem \ref{nikol}, $\WW^{1,p}_{\rm loc} \hookrightarrow \BB^{\alpha ,q}_{\infty}$
boundedly, provided $\alpha = 1-n(\frac{1}{p}-\frac{1}{q})$. The condition (\ref{qcond}) on $q$
guarantees that $\alpha \in (0,1]$. Divide (\ref{step}) by $|h|^{\alpha}$, and infer from
the arbitrariness of the ball $B$, the direction $s$ and the increment $h$, that $(V(Du_{k}))$
is bounded in $\BB^{\frac{\alpha}{2} ,2}_{\infty}$ locally on $\Omega$.
Now by version (i) of the Sobolev Embedding stated in Theorem \ref{nikol}, we have that
$\BB^{\frac{\alpha}{2} ,2}_{\infty} \hookrightarrow \LL^{r}_{\rm loc}$ boundedly for each $r< \frac{2n}{n-\alpha}$. Therefore
$(Du_{k})$ is bounded in $\LL^{r}_{\rm loc}$ for each $r< \frac{np}{n-\alpha}$ and hence $(u_{k})$ is bounded
in $\WW^{1,r}_{\rm loc}$ for each $r< \frac{np}{n-\alpha}$.

We can now repeat the above estimation using this improved bound on $(u_{k})$. The details
are as follows. Put
$$
p_{0}:= p, \quad \quad p_{j}:= \frac{np}{n-1+n(\frac{1}{p_{j-1}}-\frac{1}{q})}
$$
for $j \in \N$. Observe that we can rewrite the exponent $\bar{p}$ at (\ref{myst}) as
$$
\bar{p}= \frac{np}{n-\frac{p}{p-1}(1-n(\frac{1}{p}-\frac{1}{q}))} = \frac{n(p-1)}{n-1-\frac{n}{q}},
$$
and that, for $p_{j-1} < \bar{p}$, we have $p_{j-1} < p_{j} < \bar{p}$.
Because
$$
\bar{p} > p \quad \mbox{ precisely when } \quad q < p^{\ast},
$$
a straightforward calculation yields that
$$
p_{j} \nearrow \bar{p} \quad \mbox{ as } \quad j \nearrow \infty .
$$
With these observations in place we  apply the above difference--quotient argument
to deduce that if $(u_{k})$ is bounded in $\WW^{1,r}_{\rm loc}$ for each $r< p_{j-1}$ and
$p \leq p_{j-1} \leq q$, then it is also bounded in $\WW^{1,r}_{\rm loc}$ for each $r<p_{j}$. In view
of Remark \ref{remarque} below it follows that, for $q< \frac{np}{n-1}$, the sequence $(u_{k})$
is bounded in $\WW^{1,q}_{\rm loc}$, while for $\frac{np}{n-1} \leq q < p^{\ast}$ it
is bounded in $\WW^{1,r}_{\rm loc}$ for all $r < \bar{p}$. The conclusion follows
easily from this.
\end{proof}

\begin{remark}\label{remarque}
We record that
$$
p < \frac{np}{n-\frac{p}{p-1}\Bigl( 1-n(\frac{1}{p}-\frac{1}{q}) \Bigr)} < q \quad \mbox{ when }
\quad \frac{np}{n-1}<q<p^{\ast}
$$
and
$$
\frac{np}{n-\frac{p}{p-1}\Bigl( 1-n(\frac{1}{p}-\frac{1}{q}) \Bigr)} \geq q \quad \mbox{ when }
\quad p \leq q \leq \frac{np}{n-1} \quad \mbox{ and } \quad q > \frac{n}{n-1}.
$$
Hence there is integrability improvement locally in $\Omega$ of $F$--minimizers for the full range
of exponents $q$ satisfying (\ref{qcond}). Furthermore, $\bar{p}=\frac{np}{n-1}$ when $q = \frac{np}{n-1}$,
and $\bar{p}=\bar{p}(q)$ is decreasing as a function of $q$ with
$$
\left\{
\begin{array}{ll}
\bar{p} \searrow p \mbox{ as } q \nearrow p^{\ast} & \mbox{ when } 1<p<n\\
{} & {}\\
\bar{p} \searrow \frac{n(p-1)}{n-1} \mbox{ as } q \nearrow \infty & \mbox{ when } p \geq n,
\end{array}
\right.
$$
where we remark that $\frac{n(p-1)}{n-1} \geq p$ for $p \geq n$ with equality precisely when
$p=n$.
\end{remark}

\section{Proof of Theorem \ref{th2}}

\noindent
Throughout this section $u \in \WW^{1,p}_{\loc}(\Omega , \RN ) $
denotes a local $F$--minimizer. For the sake of simplicity,
we shall give the proof in case the integrand $F \colon \M \to \real$ is $\CC^2$ and
satisfies the hypotheses (H1) and (H2'), with $q<\frac{pn}{n-1}$.
The general case can be treated by a suitable  approximation argument, inspired by \cite{ff}
and \cite{elm1}, and also sketched in \cite{CKPdN}.

\noindent
Our aim is to show that $V(Du) \in \WW^{1,2}_{\rm loc}(\Omega ,\M )$, where we recall the definition
of the auxiliary functions as
$$
V( \xi ) := \langle \xi \rangle^{\frac{p-2}{2}}\xi , \quad \quad \langle \xi \rangle
:= \sqrt{\mu^2+|\xi |^{2}}.
$$
For later reference we note that for a $\CC^2$ map $w$ a routine calculation yields
\begin{equation}\label{prelim}
\Bigl| D\Bigl[ V(Dw) \Bigr] \Bigr|^{2} \leq \Bigl( \frac{|p-2|}{2}+1 \Bigr)^{2}
\langle Dw \rangle^{p-2} |D^{2}w|^{2}
\end{equation}
Before proceeding with the proof of Theorem \ref{th2}, we need to carry out an approximation
procedure, which is essentially based on the arguments contained in \cite{CKPdN}. Here we give a
version suitable for our needs, partly for the sake of completeness and partly because the present
set--up differs slightly from that of \cite{CKPdN}.

\noindent
Fix a subdomain with a smooth boundary $\Omega^{\prime} \Subset \Omega$ and take   $k \in \N$,  so large
that we have the continuous embedding
$\WW^{k,2}(\Omega^{\prime}) \hookrightarrow \CC^{2}(\overline{\Omega^{\prime}})$.
For a smooth kernel $\phi \in \CC^{\infty}_{c}(B_{1}(0))$ with $\phi \geq 0$ and $\int_{B_{1}(0)}
\! \phi = 1$, we consider the corresponding family of mollifiers $( \phi_{\eps})_{\eps >0}$
and put $\tilde{u}_{\eps} := \phi_{\eps} \ast u$ on $\Omega^{\prime}$ for each positive
$\eps < \dist (\Omega^{\prime},\partial \Omega )$. By Theorem \ref{th1}, we have that
$Du\in L^{q}_{\mathrm{loc}}$ and hence
\begin{equation}\label{61}
\tu \to u \mbox{ as } \eps \searrow 0 \mbox{ strongly in } \WW^{1,q}(\Omega^{\prime}, \RN )\,.
\end{equation}
Moreover we remark that, for a suitable function $\tilde{\eps}= \tilde{\eps}(\eps )$ with
$\tilde{\eps} \searrow 0$ as $\eps \searrow 0$, we also have
\begin{equation}\label{62}
\tilde{\eps} \int_{\Omega^{\prime}} \! |D^{k}\tu |^{2} \to 0 \mbox{ as } \eps \searrow 0.
\end{equation}
For small $\eps > 0$, we let $\ue \in \WW^{k,2}(\Omega^{\prime}) \cap \WW^{1,p}_{\tu}(\Omega^{\prime})$
denote a minimizer to the functional
$$
v \mapsto \int_{\Omega^{\prime}} \! \Bigl(  F(Dv) +
\frac{\teps}{2}|D^{k}v|^{2} \Bigr)
$$
on the Sobolev class $\WW^{k,2}(\Omega^{\prime}) \cap \WW^{1,p}_{\tu}(\Omega^{\prime})$.
The existence of $\ue$ is easily established by the direct method.
Next two Lemmas are proven in \cite{CKPdN} (see Lemmas 8 and  9 there) in a more general version.
Here we state them in the form needed for our aims.

\begin{lemma}\label{lem61}
For each $\varphi \in \WW^{k,2}(\Omega^{\prime}) \cap \WW^{1,p}_{0}(\Omega^{\prime})$,
\begin{equation}\label{el}
0 = \int_{\Omega^{\prime}} \! \Bigl( \langle F^{\prime}(D\ue ),D\varphi \rangle +\teps
\langle D^{k}\ue , D^{k}\varphi \rangle \Bigr).
\end{equation}
Furthermore, $\ue \in \WW^{2k,2}_{\rm loc}(\Omega^{\prime})$.
\end{lemma}

\bigskip

\begin{lemma}\label{lem62}
As $\eps \searrow 0$, we have that
$$\qquad\qquad\int_{\Omega^{\prime}}|D\ue -D u|^{ p}\,\dx\to 0$$
and
$$
\int_{\Omega^{\prime}} \! F(D\ue )\,\dx \to \int_{\Omega^{\prime}} \! F(Du)\,\dx.
$$
\end{lemma}

\noindent
We are now ready to prove Theorem \ref{th2}.


\begin{proof}[Proof of Theorem \ref{th2}]
Fix $B_{2R} = B_{2R}(x_{0}) \subset \Omega^{\prime}$,
radii $R \leq r < s  \leq 2R \leq 2$ and a smooth
cut-off function $\rho$ satisfying $1_{B_{r}} \leq \rho \leq 1_{B_{s}}$ and
$|D^{i}\rho | \leq \Bigl( \frac{2}{s-r} \Big)^{i}$ for each $i \in \N$.
According to Lemma \ref{lem61}, we can test the Euler--Lagrange system \eqref{el} with
$\varphi = \rho^{2k} D^{2}_{j}\ue$, for each direction $1 \leq j \leq n$, thus getting
\begin{eqnarray}\label{1to4}
0 &=& \int_{\Omega^{\prime}} \! \Big\langle F^{\prime}(D\ue ),D_{j}^{2}D\ue \Big\rangle \rho^{2k}
+ \int_{\Omega^{\prime}} \! \Big\langle F^{\prime}(D\ue ),D_{j}^{2}\ue \otimes D\Bigl( \rho^{2k} \Bigr)
\Big\rangle \nonumber\\
&&  + \teps \int_{\Omega^{\prime}} \! \Big\langle D^{k}\ue , D^{k} \Bigl( D_{j}^{2}\ue \rho^{2k} \Bigr)
\Big\rangle \nonumber\\
&=:& I + II +  III.
\end{eqnarray}
Integration by parts yields
\begin{eqnarray*}
I &=& -\int_{\Omega^{\prime}} \! \left( \rho^{2k}F^{\prime \prime}(D\ue )\Big[ D_{j}D\ue , D_{j}D\ue\Big]\right)
-\int_{\Omega^{\prime}} \! \left( 2k\Big\langle \rho^{2k-1}D_{j}\rho F^{\prime}(D\ue ) , D_{j}D\ue \Big\rangle \right)\\
&\le &-\int_{\Omega^{\prime}} \!  \rho^{2k}\wei^{p-2}|D_{j}D\ue |^{2}+2k\int_{\Omega^{\prime}} \!
 \rho^{2k-1}|D_{j}\rho| \wei^{q-1} |D_{j}D\ue|
\end{eqnarray*}
where we used assumptions (H2') and  \eqref{(H4)} . Hence, using Young's inequality in the last  integral, we obtain
\begin{eqnarray}\label{i}
I &\leq& -\int_{\Omega^{\prime}} \! \rho^{2k}\wei^{p-2}|D_{j}D\ue |^{2}
+\frac{1}{2}\int_{\Omega^{\prime}} \! \rho^{2k}\wei^{p-2}|D_{j}D\ue |^{2}\cr\cr
&&+c(p,k)\int_{\Omega^{\prime}} \! \rho^{2(k-1)}|D_{j}\rho |^{2}\wei^{2q-p}\cr\cr
&\le &- \frac{1}{2}\int_{\Omega^{\prime}} \! \rho^{2k}\wei^{p-2}|D_{j}D\ue |^{2}+c(p,k)\int_{\Omega^{\prime}} \! \rho^{2(k-1)}|D_{j}\rho |^{2}\wei^{2q-p}.
\end{eqnarray}
Similarly, by virtue of \eqref{(H4)} and Cauchy--Schwarz' inequality, we get
\begin{eqnarray}\label{ii}
II &\leq& c(p,k)\int_{\Omega^{\prime}} \! \wei^{q-1}\rho^{2k-1}|D\rho | |D_{j}^{2}\ue |\cr\cr
&\leq& c(p,k)\int_{\Omega^{\prime}}  \wei^{2q-p}\rho^{2(k-1)}|D\rho |^{2}+
\frac{1}{4}\int_{\Omega^{\prime}} \! \rho^{2k}\wei^{p-2}|D_{j}D\ue |^{2}.
\end{eqnarray}

\noindent In order to estimate  $III$, we argue as in \cite{CKPdN} writing
$$
III = \teps \int_{\Omega^{\prime}} \! \Big\langle D^{k}\ue , D_{j}D^{k} \Bigl( \rho^{2k}D_{j}\ue \Bigr) -D^{k}
\Bigl( D_{j}\bigl(\rho^{2k}\bigr)D_{j}\ue \Bigr) \Big\rangle
$$
and integrating the first term by parts,
\begin{eqnarray*}
III &=& -\teps \int_{\Omega^{\prime}} \! \Bigl( \Big\langle D_{j}D^{k}\ue , D^{k}\Bigl( \rho^{2k}D_{j}\ue \Bigr)
\Bigr\rangle  -\teps \int_{\Omega^{\prime}} \! \Big\langle D^{k}\ue , D^{k}\Bigl( D_{j}\bigl( \rho^{2k} \bigr)
D_{j}\ue \Bigr) \Bigr\rangle \Bigr)\\
 &=:& III_{1}+III_{2}.
\end{eqnarray*}
We estimate these terms by use of Cauchy--Schwarz' inequality, Leibniz' product formula and the assumptions
on $D^{i}\rho$ (simplifying also by use of $s-r \leq 1$):
\begin{eqnarray*}
III_{1} &\leq& -\teps \int_{\Omega^{\prime}} \! \rho^{2k} |D_{j}D^{k}\ue |^{2} + \frac{c_{k}\teps}{(s-r)^{k}}
\int_{\Omega^{\prime}} \! \rho^{k}|D_{j}D^{k}\ue | \sum_{i=0}^{k-1} |D^{i}D_{j}\ue |\\
&\leq& -\frac{2\teps}{3}\int_{\Omega^{\prime}} \! \rho^{2k} |D_{j}D^{k}\ue |^{2} +\frac{c_{k}\teps}{(s-r)^{2k}}
\int_{B_{2R}} \! \Bigl( \sum_{i=0}^{k-1} |D^{i}D_{j}\ue | \Bigr)^{2}\\
&\leq& -\frac{2\teps}{3}\int_{\Omega^{\prime}} \! \rho^{2k} |D_{j}D^{k}\ue |^{2} +\frac{c_{k}\teps}{(s-r)^{2k}}
\int_{B_{2R}} \! \sum_{i=0}^{k-1} |D^{i}D_{j}\ue |^{2}
\end{eqnarray*}
for a (new) constant $c_{k}$. Likewise,
$$
III_{2} \leq \frac{\teps}{3}\int_{\Omega^{\prime}} \! \rho^{2k} |D_{j}D^{k}\ue |^{2} +\frac{c_{k}\teps}{(s-r)^{2k+2}}
\int_{B_{2R}} \! \left(\sum_{i=0}^{k-1} |D^{i}D_{j}\ue |^{2}+|D^k\ue|^2\right),
$$
where we remark that the increased power of the factor $(s-r)$ is due to the presence of an additional
$D_{j}$-derivative on $\rho^{2k}$ in $III_2$. Collecting the above bounds and adjusting the constant $c_k$ we
arrive at
\begin{equation}\label{iv}
III \leq -\frac{\teps}{3}\int_{\Omega^{\prime}} \!\rho^{2k} |D_{j}D^{k}\ue |^{2} +\frac{c_{k}\teps}{(s-r)^{2k+2}}
\int_{B_{2R}} \! \left(\sum_{i=0}^{k-1} |D^{i}D_{j}\ue |^{2}+|D^k\ue|^2\right).
\end{equation}
Inserting the bounds (\ref{i}), (\ref{ii}),  (\ref{iv}) in \eqref{1to4} and using the properties of $\rho$
we get for each $1 \leq j \leq n$:
\begin{eqnarray*}
&&\frac{1}{4}\int_{\Omega^{\prime}} \! \rho^{2k}\wei^{p-2}|D_{j}D\ue |^{2}+ \frac{\teps}{3}\int_{\Omega^{\prime}} \!\rho^{2k} |D_{j}D^{k}\ue |^{2}\\
&&\leq \frac{c(p,k)}{(s-r)^{2}}\int_{B_{s}\setminus B_{r}} \! \wei^{2q-p} + \frac{c\teps}{(s-r)^{2k+2}}\int_{B_{2R}} \!\left(
\sum_{i=0}^{k-1} |D_{j}D^{i}\ue |^{2}+|D^k\ue|^2\right).
\end{eqnarray*}
Adding up these inequalities over $j \in \{ 1, \, \dots \, , \, n \}$ and adjusting
the constants we arrive at
\begin{eqnarray}
& &\int_{\Omega^{\prime}} \! \rho^{2k}\wei^{p-2}|D^2\ue |^{2} +
\frac{4\teps}{3}\int_{\Omega^{\prime}} \! \rho^{2k}|D^{k+1}\ue |^{2}\nonumber\\
&&\leq \frac{c(n,p,k)}{(s-r)^{2}}\int_{B_{s}\setminus B_{r}}  \wei^{2q-p} +
 \frac{A(\eps )}{(s-r)^{2k+2}},
\end{eqnarray}
where   $A(\eps )$ is
independent of $r$, $s$ and where, by a suitable version of the Gagliardo--Nirenberg interpolation inequality,
$$
 A(\eps ) \to 0 \quad \mbox{ as } \quad \eps \searrow 0.
$$
Omitting the second term, involving $(k+1)$-th order derivatives, on the left--hand side, the above inequality reduces to
\begin{equation}\label{hek9}
\int_{\Omega^{\prime}} \! \rho^{2k}\wei^{p-2}|D^2\ue |^{2} \leq
\frac{c(n,p,k)}{(s-r)^{2}}\int_{B_{s}\setminus B_{r}}  \wei^{2q-p}  + \frac{A(\eps )}{(s-r)^{2k+2}}.
\end{equation}
Now, taking into account estimate \eqref{prelim}, an elementary calculation implies that
\begin{equation*}\left|D\left(\rho^k V(D\ue)\right)\right|^{2}
\le c(p)\left[\rho^{2k}\wei^{p-2}|D^2\ue|^2+  k^2\rho^{2k-2}|D\rho|^2|V(D\ue)|^{2}\right],
\end{equation*}
Therefore, by virtue of estimate \eqref{hek9} and by the Sobolev Embedding Theorem, we obtain
\begin{eqnarray}\label{hek10}
&&\left(\int_{\Omega^{\prime}} \! \left|\rho^k V(D\ue)\right|^{\frac{2n}{n-2}} \right)^{\frac{n-2}{n}}\le
c\int_{\Omega^{\prime}} \! \left|D\left(\rho^k V(D\ue)\right)\right|^{2} \cr\cr
&\leq&
\frac{c(n,N,p,k)}{(s-r)^{2}}\int_{B_{s}\setminus B_{r}}  \wei^{2q-p}
+ \frac{c(n,N,p,k)}{(s-r)^{2}}\int_{B_{s}\setminus B_{r}}  |V(D\ue)|^{2}+ \frac{A(\eps )}{(s-r)^{2k+2}}\cr\cr
&\leq&
\frac{c(n,N,p,k)}{(s-r)^{2}}\int_{B_{s}\setminus B_{r}\cap\{|D\ue|\le \mu\}}  \wei^{2q-p}+\frac{c(n,N,p,k)}{(s-r)^{2}}
\int_{B_{s}\setminus B_{r}\cap\{|D\ue|> \mu\}}  \wei^{2q-p}\cr\cr
&&  +\frac{c(n,N,p,k)}{(s-r)^{2}}\int_{B_{s}\setminus B_{r}}  |V(D\ue)|^{2}+ \frac{A(\eps )}{(s-r)^{2k+2}}\cr\cr
&\leq&
\frac{c(n,N,p,k,\mu)}{(s-r)^{2}}\int_{B_{s}\setminus B_{r}}  (1+|V(D\ue)|^2)^{\frac{2q-p}{p}}
+ \frac{c(n,N,p,k)}{(s-r)^{2}}\int_{B_{s}\setminus B_{r}}  |V(D\ue)|^{2}\cr\cr
&& +\frac{A(\eps )}{(s-r)^{2k+2}}
\end{eqnarray}
 We can write
$$
\frac{p}{2q-p} = \frac{\theta}{\frac{n}{n-2}}+\frac{1-\theta}{\frac{q}{p}},
$$
where, since $p<q<p\frac{n}{n-1}$,
$$
\theta = \frac{q-p}{2q-p}\times \frac{pn}{pn-q(n-2)} \in (0,1)
$$
(note that the case $p=q$ doesn't require these arguments).
H\"{o}lder's inequality yields
$$\int_{B_{s}\setminus B_{r}} |V(D\ue)|^{\frac{2(2q-p)}{p}}\le \left(\int_{B_{s}\setminus B_{r}}
|V(D\ue)|^{\frac{2n}{n-2}}\right)^{\frac{\theta(n-2)}{n}\frac{2q-p}{p}}\left(\int_{B_{s}\setminus B_{r}}
|V(D\ue)|^{\frac{2q}{p}}\right)^{\frac{(1-\theta)(2q-p)}{q}}$$
Inserting the previous inequality in \eqref{hek10}, we obtain
\begin{eqnarray}\label{hek11}
&&\int_{\Omega^{\prime}} \! \left|\rho^k V(D\ue)\right|^{\frac{2n}{n-2}}\cr\cr
&\leq&
\frac{c(n,N,p,k,\mu)}{(s-r)^{\frac{2n}{n-2}}}\left(\int_{B_{s}\setminus B_{r}}
|V(D\ue)|^{\frac{2n}{n-2}}\right)^{\theta\frac{2q-p}{p}}\left(\int_{B_{s}\setminus B_{r}}
|V(D\ue)|^{\frac{2q}{p}}\right)^{\frac{(1-\theta)(2q-p)}{q}\frac{n}{n-2}}\cr\cr
&& +\frac{c(n,N,p,k)}{(s-r)^{\frac{2n}{n-2}}}\left(\int_{B_{s}\setminus B_{r}}
|V(D\ue)|^{2}\right)^{\frac{n}{n-2}}+ \frac{c(n,N,p,k,\mu)}{(s-r)^{\frac{2n}{n-2}}}|B_{s}\setminus B_{r}|^{\frac{n}{n-2}}\cr\cr
&& +\frac{\tilde A(\eps )}{(s-r)^{\frac{(2k+2)n}{n-2}}}
\end{eqnarray}
where we set $\tilde A(\eps)=(A(\eps))^{\frac{n}{n-2}}$. Since
$$
\frac{2q-p}{p}\theta=\frac{(q-p)n}{pn-q(n-2)}<1
$$
for $q<\frac{pn}{n-1}$, it is legitimate to apply Young's inequality with the pair of conjugate exponents
$$
d= \frac{pn-q(n-2)}{(q-p)n}\qquad \mbox{ and } \qquad d'=\frac{1}{2}\frac{pn-q(n-2)}{pn-q(n-1)}
$$
in the second line of \eqref{hek11}, thus getting
\begin{eqnarray}\label{hek12}
&&\int_{B_r} \! \left|V(D\ue)\right|^{\frac{2n}{n-2}}\le\int_{\Omega^{\prime}} \! \left|\rho^k V(D\ue)\right|^{\frac{2n}{n-2}}\cr\cr
&\leq&\frac{1}{2}\left(\int_{B_{s}\setminus B_{r}}  |V(D\ue)|^{\frac{2n}{n-2}}\right)+
\frac{c(n,N,p,k,\mu)}{(s-r)^{d'\frac{2n}{n-2}}}
\left(\int_{B_{s}\setminus B_{r}} |V(D\ue)|^{\frac{2q}{p}}\right)^{\frac{(p-q)n+2q-p}{pn-q(n-1)}\frac{n}{n-2}}\cr\cr
&& +\frac{c(n,N,p,k)}{(s-r)^{\frac{2n}{n-2}}}\left(\int_{B_{s}\setminus B_{r}}
|V(D\ue)|^{2}\right)^{\frac{n}{n-2}}+ \frac{c(n,N,p,k,\mu)}{(s-r)^{\frac{2n}{n-2}}}|B_{s}\setminus B_{r}|^{\frac{n}{n-2}}\cr\cr
&& +\frac{\tilde A(\eps )}{(s-r)^{\frac{(2k+2)n}{n-2}}}
\end{eqnarray}
As this estimate is valid for all radii $R\leq r < s \leq 2R$, we can apply the
hole--filling method of Widman. This yields in the usual manner
\begin{eqnarray}\label{hek13}
&&\int_{B_R} \! \left|V(D\ue)\right|^{\frac{2n}{n-2}}\le
\frac{c(n,N,p,k,\mu)}{R^{d'\frac{2n}{n-2}}}\left(\int_{B_{2R}}  |V(D\ue)|^{\frac{2q}{p}}\right)^{\frac{(p-q)n+2q-p}{pn-q(n-1)}\frac{n}{n-2}}\cr\cr
&& +\frac{c(n,N,p,k)}{R^{\frac{2n}{n-2}}}\left(\int_{B_{2R}}  |V(D\ue)|^{2}\right)^{\frac{n}{n-2}}+
c(n,N,p,k,\mu)R^n+ \frac{\tilde A(\eps )}{R^{\frac{(2k+2)n}{n-2}}}
\end{eqnarray}
From estimate \eqref{hek13}, through the higher integrability of Theorem \ref{th1}, it follows that
$V(D\ue)$ is bounded in $\LL^{\frac{2n}{n-2}}(B_{R}, \M )$ uniformly as $\varepsilon \searrow 0$ and
so, by the arbitrariness of the ball $B_{2R}(x_{0}) \subset \Omega^{\prime}$
and a simple covering argument, we conclude that $V(D\ue)$ is bounded in
$\LL^{\frac{2n}{n-2}}_{\rm loc}(\Omega^{\prime},\M )$. In view of (\ref{prelim}) and
(\ref{hek9}) it then also follows that $(V(D\ue ))$ is bounded in
$\WW^{1,2}_{\rm loc}(\Omega^{\prime}, \M )$ uniformly as $\varepsilon \searrow 0$.
Therefore, we conclude  by passing to the limits as $\eps \searrow 0$, using also compactness of the
Sobolev embedding on the right--hand side and Fatou's Lemma on the left--hand side, that
\begin{eqnarray}\label{hek14}
&&\int_{B_R} \! \left|V(Du)\right|^{\frac{2n}{n-2}}\le
\frac{c(n,N,p,\mu)}{R^{d'\frac{2n}{n-2}}}\left(\int_{B_{2R}} |V(Du)|^{\frac{2q}{p}}\right)^{\frac{(p-q)n+2q-p}{pn-q(n-1)}\frac{n}{n-2}}\cr\cr
&& +\frac{c(n,N,p)}{R^{\frac{2n}{n-2}}}\left(\int_{B_{2R}}  |V(Du)|^{2}\right)^{\frac{n}{n-2}}+ c(n,N,p,\mu)R^n
\end{eqnarray}
and
\begin{eqnarray}\label{hek15}
&&\int_{B_R} \! \left|D(V(Du))\right|^{2}\le
\frac{c(n,N,p,\mu)}{R^{d'\frac{2n}{n-2}}}\left(\int_{B_{2R}} |V(Du)|^{\frac{2q}{p}}\right)^{\frac{(p-q)n+2q-p}{pn-q(n-1)}\frac{n}{n-2}}\cr\cr
&& +\frac{c(n,N,p)}{R^{\frac{2n}{n-2}}}\left(\int_{B_{2R}}  |V(Du)|^{2}\right)^{\frac{n}{n-2}}+ c(n,N,p,\mu)R^n
\end{eqnarray}
\end{proof}

\noindent
{\bf Acknowledgments}  Parts of the research were done while MC and APdN were
visiting the Oxford Centre for Nonlinear PDE, and while JK was visiting Dept.~Maths.
`R.~Caccioppoli' in Napoli. We wish to thank both institutions for financial support and
hospitality. The work was supported by the EPSRC Science and Innovation award to the Oxford
Centre for Nonlinear PDE (EP/E035027/1), and JK was also partially supported by ERC grant
207573 `Vectorial Problems'.

\noindent
Universit\`{a} del Sannio, Piazza Arechi II - Palazzo De Simone, 82100 Benevento, Italy

\noindent
{\em E-mail address}: carozza@unisannio.it
\medskip

\noindent
Mathematical Institute, University of Oxford, 24--29 St.~Giles', Oxford OX1 3LB, England

\noindent
{\em E-mail address}: kristens@maths.ox.ac.uk
\medskip

\noindent
Universit\`{a} di Napoli `Federico' Dipartimento di Mat.~e Appl. 'R.~Caccioppoli',
Via Cintia, 80126 Napoli, Italy

\noindent
{\em E-mail address}: antpassa@unina.it

\end{document}